\documentclass[a4paper,12pt, reqno]{amsart}
\usepackage{a4}

\usepackage{amsfonts,exscale}
\usepackage{mleftright}

\usepackage{srcltx,hyperref}
\usepackage{cite}

\usepackage{color}

\long\def\forget#1{}

\usepackage{amsmath}
\usepackage{amssymb}
\usepackage{amscd}
\usepackage{textcomp}
\usepackage{amsthm}
\theoremstyle{plain}
\newtheorem{Lemma}{Lemma}[section]
\newtheorem{Theorem}[Lemma]{Theorem}

\theoremstyle{definition}

\newcommand{\DS}{\displaystyle}

\newcommand{\es}{\enspace}

\newcommand{\dpl}{{\mathchoice{\mbox{\rm (\hspace{-0.15em}(}}
{\mbox{\rm (\hspace{-0.15em}(}}
{\mbox{\scriptsize\rm (\hspace{-0.15em}(}}
{\mbox{\tiny\rm (\hspace{-0.15em}(}}}}
\newcommand{\dpr}{{\mathchoice{\mbox{\rm )\hspace{-0.15em})}}
{\mbox{\rm )\hspace{-0.15em})}}
{\mbox{\scriptsize\rm )\hspace{-0.15em})}}
{\mbox{\tiny\rm )\hspace{-0.15em})}}}}
\newcommand{\invlim}[1][]{\ifthenelse{\equal{#1}{}}
{\DS \lim_{\longleftarrow}}
{\DS \lim_{\underset{#1}{\longleftarrow}}}
}
\newcommand{\dirlim}[1][]{\ifthenelse{\equal{#1}{}}
{\DS \lim_{\longrightarrow}}
{\DS \lim_{\underset{#1}{\longrightarrow}}}
}

\newcommand{\BOne} {{\mathchoice{\hbox{\rm1\kern-2.7pt l\kern.9pt}}
{\hbox{\rm1\kern-2.7pt l\kern.9pt}}
{\hbox{\scriptsize\rm1\kern-2.3pt l\kern.4pt}}
{\hbox{\scriptsize\rm1\kern-2.4pt l\kern.5pt}}}}

\newcommand{\BC}{{\mathbb{C}}}

\newcommand{\BF}{{\mathbb{F}}}

\newcommand{\BN}{{\mathbb{N}}}

\newcommand{\sS}{{\mathscr{S}}}

\usepackage{mathrsfs}

\newcommand{\CA}{{\mathcal{A}}}

\newcommand{\CF}{{\mathcal{F}}}
\newcommand{\CG}{{\mathcal{G}}}

\newcommand{\CK}{{\mathcal{K}}}

\newcommand{\CP}{{\mathcal{P}}}

\newcommand{\CR}{{\mathcal{R}}}
\newcommand{\CS}{{\mathcal{S}}}

\newcommand{\FN}{{\mathfrak{N}}}

\newcommand{\FP}{{\mathfrak{P}}}

\newcommand{\Fg}{{\mathfrak{g}}}

\newcommand{\Fw}{{\mathfrak{w}}}

\DeclareMathOperator{\Tr}{Tr}

\begin{document}
\title{Divisibility problems for function fields}
\author{Stephan Baier, Arpit Bansal, Rajneesh Kumar Singh}
\address{Stephan Baier, Ramakrishna Mission Vivekananda Educational Research Institute, Department of Mathematics, 
G. T. Road, PO Belur Math, Howrah, West Bengal 711202, 
India; email: email$_{-}$baier@yahoo.de}
\address{Arpit Bansal, Jawaharlal Nehru University, School of Physical Sciences, New-Delhi 110067, India; email: apabansal@gmail.com}
\address{Rajneesh Kumar Singh, Ramakrishna Mission Vivekananda Educational Research Institute, G. T. Road, PO Belur Math, Howrah, 
West Bengal 711202, India; rajneeshkumar.s@gmail.com}
\subjclass[2010]{11B75;11N36;11L40;11P99}
\begin{abstract} We investigate three combinatorial problems considered by Erd\"os, Rivat, Sar\-k\"ozy and Sch\"on regarding divisibility properties of 
sum sets and sets of shifted products of integers in the context of function fields. Our results in this function field setting are
better than those previously obtained for subsets of the integers. These improvements depend on a version of the large sieve for sparse
sets of moduli developed recently by the first and third-named authors.
\end{abstract}
\maketitle

\tableofcontents
\section{Introduction and main results}
In this paper, we investigate three combinatorial problems considered by Erd\"os, Rivat, Sar\-k\"ozy and Sch\"on regarding divisibility properties of 
sum sets and sets of shifted products of integers in the context of function fields. The results we obtain 
in this function field setting are comparably better than those obtained in the case of integers. We first state the best known
results on these problems in the integer setting. 

\begin{Theorem}[Theorem in \cite{Bai}] \label{Psettheo} 
We call a subset $\mathcal{S}$ of the positive integers a $\mathcal{P}$-set if no element of $\mathcal{S}$ divides the sum of two 
{\rm (}not necessarily distinct{\rm )} larger elements of $\mathcal{S}$. Let $\mathcal{S}$ be a $\mathcal{P}$-set of pairwise coprime positive integers and 
denote by 
$\mathcal{A}_{\mathcal{S}}(N)$ the number of elements of $\mathcal{S}$ not exceeding $N$. 
If $\varepsilon>0$, then there exist infinitely many positive integers $N$ such that
$$
\mathcal{A}_{\mathcal{S}}(N)\le (3+\varepsilon)N^{2/3}(\log N)^{-1}.
$$
\end{Theorem}

\begin{Theorem}[Theorem 3 in \cite{ErS}] \label{ErSatheo} If $N\in \mathbb{N}$ is sufficiently large, $\mathcal{A}\subseteq \{1,...,N\}$ and 
$a+a'$ is square-free for all $a,a'\in \mathcal{A}$, then 
$$
\sharp(\mathcal{A})\le 3N^{3/4}\log N. 
$$
\end{Theorem}

\begin{Theorem}[Theorem 4.2. in \cite{SaR}] \label{SarRitheo} If $N\in \mathbb{N}$ is sufficiently large, $\mathcal{A},\mathcal{B}\subseteq \{1,...,N\}$ and 
$ab+1$ is square-free for all $a\in \mathcal{A}$ and $b\in \mathcal{B}$, then 
$$
\min\{\sharp(\mathcal{A}),\sharp(\mathcal{B})\}\le 2N^{3/4}\log N. 
$$
\end{Theorem}

For the history of these problems, see the papers \cite{Bai}, \cite{ErS} and \cite{SaR}. It is likely that the exponents in
the estimates in the above theorems are not optimal. We note that the exponent $2/3$ in Theorem \ref{Psettheo} cannot be replaced by any value
smaller than $1/2$, however. In fact, Erd\"os and  Sar\-k\"ozy \cite{ErS} gave a simple example of a $\mathcal{P}$-set $\mathcal{S}$ satisfying
$$
\mathcal{A}_{\mathcal{S}}(N)\gg N^{1/2}(\log N)^{-1}
$$
as $N \rightarrow \infty$, namely the set of squares of primes $p \equiv 3 
\bmod {4}$. By a more elaborate construction, Elsholtz and Planitzer \cite{ElP} established the existence of a $\mathcal{P}$-set $\mathcal{S}$ satisfying
$$
\mathcal{A}_{\mathcal{S}}(N)\gg N^{1/2}(\log N)^{-1/2}(\log \log N)^{-2} (\log \log \log N)^{-2} 
$$
as $N \rightarrow \infty$. To date, this is the best known lower bound. It would 
be interesting to produce such lower bounds in the function field setting as well, but we will not consider this problem here.  

The proofs of Theorems \ref{Psettheo}, \ref{ErSatheo} and \ref{SarRitheo} crucially depend on the large sieve. In all three proofs, the sets of relevant moduli in the large sieve are sparse,
but it turns out that the results developed by L. Zhao and the first-named author of this paper on the large sieve with sparse sets of
moduli (see \cite{Bai1}, \cite{BaZh} and \cite{Zha}) are not sufficient to obtain any improvement over what can be established using the large sieve with full sets of moduli. 
However, the situation is different in the function field setting. Recently, the first and third-named authors of the present paper established 
an essentially best possible large sieve inequality with sparse sets of moduli for function fields (see \cite{StephanRS}) which allows to obtain comparably better 
results for the analogues of the combinatorial problems above in the function field setting. This is the object of this paper. We prove the following
results corresponding to Theorems \ref{Psettheo}, \ref{ErSatheo} and \ref{SarRitheo}, where in the theorems below,
$\BF_q$ denotes a finite field with $q$ elements. 

\begin{Theorem} \label{Psettheoff} 
We call a subset $\mathcal{S}$ of $\BF_q[t]$ a 
$\mathcal{P}$-set if it consists of monic polynomials and no element of $\mathcal{S}$ divides the sum of two {\rm (}not necessarily distinct{\rm )} 
elements of $\mathcal{S}$ of larger
degree. Let $\mathcal{S}$ be a $\mathcal{P}$-set of pairwise coprime polynomials in $\BF_q[t]$ and 
denote by $\mathcal{A}_{\mathcal{S}}(N)$ the number of elements of $\mathcal{S}$ 
of degree not exceeding $N$. 
If $\varepsilon>0$, then there exist infinitely many positive integers $N$ such that
$$
\mathcal{A}_{\mathcal{S}}(N)\le q^{N(1/\Phi+\varepsilon)},
$$
where $\Phi$ is the golden ratio, given by
$$
\Phi:=\frac{\sqrt{5}+1}{2}.
$$
\end{Theorem}

We remark that the statement of Theorem \ref{Psettheoff} is empty if $q=2^n$ for some $n\in \mathbb{N}$ since in this case, $2f=0$ for any polynomial $f\in \BF_q[t]$. 

\begin{Theorem} \label{ErSatheoff}
If $\varepsilon>0$, $N\in \mathbb{N}$ is sufficiently large, 
$\mathcal{F}$ is a set of monic polynomials in $\BF_q[t]$ of degree not exceeding $N$ and 
$f+f'$ is square-free for all $f,f'\in \mathcal{F}$, then 
$$
\sharp(\mathcal{F})\le q^{N(2/3+\varepsilon)}.
$$
\end{Theorem}

\begin{Theorem} \label{SarRitheoff} If $\varepsilon>0$, 
$N\in \mathbb{N}$ is sufficiently large, 
$\mathcal{F},\mathcal{G}$ are sets of monic polynomials in $\BF_q[t]$ of degree not exceeding $N$ and 
$fg+1$ is square-free for all $f\in \mathcal{F}$ and $g\in \mathcal{G}$, then 
$$
\min\{\sharp(\mathcal{F}),\sharp(\mathcal{G})\}\le q^{N(2/3+\varepsilon)}. 
$$
\end{Theorem}

In the last two theorems, ``square-free'' has the obvious meaning\  :\  A monic polynomial in $\BF_q[t]$ is square-free if each factor appears 
precisely once in its unique factorization into irreducible monic polynomials. This is equivalent to saying that the polynomial has no 
multiple roots in the algebraic closure $\overline{\BF_q[t]}$ since $\BF_q$ is a perfect field.

To compare the above Theorems \ref{Psettheoff}, \ref{ErSatheoff} and \ref{SarRitheoff} on sets of polynomials with 
Theorems \ref{Psettheo}, \ref{ErSatheo} and \ref{SarRitheo} on sets of integers, we note that the term $q^N$ in the estimates in Theorems 
\ref{Psettheoff}, \ref{ErSatheoff} and \ref{SarRitheoff} takes the rule of $N$ in the estimates in 
Theorems \ref{Psettheo}, \ref{ErSatheo} and \ref{SarRitheo}. Thus, the exponent $2/3$ in Theorem \ref{Psettheo} is replaced by the exponent
$1/\Phi=0.61...$ in Theorem \ref{Psettheoff}, and the exponent $3/4$ in Theorems \ref{ErSatheo} and \ref{SarRitheo} is replaced by the
exponent $2/3$ in Theorems \ref{ErSatheoff} and \ref{SarRitheoff}.

Our proofs of Theorems \ref{Psettheoff}, \ref{ErSatheoff} and \ref{SarRitheoff} follow closely those of Theorems \ref{Psettheo},  
\ref{ErSatheo} and \ref{SarRitheo} in \cite{Bai}, \cite{ErS} and \cite{SaR}, respectively.\\ \\
{\bf Acknowledgement:} We would like to thank Igor Shparlinski for making us aware of Theorems \ref{ErSatheo} and \ref{SarRitheo} and
suggesting to consider them in the function field setting. 

\section{Notations} 
We begin by recalling some standard notations and facts about function fields.
Let $\BF_q$ be a fixed finite field with $q$ elements of characteristic $p$ and let 
$$ 
\Tr \  :\  \BF_q \to \BF_p
$$
be the trace map.
Let $\BF_q(t)_\infty$ be the completion of $\BF_q(t)$ at $\infty$ (i.e. $\BF_q\dpl 1/t\dpr$). 
The absolute value $|\cdot |_\infty$ of $\BF_q(t)_\infty$ is defined by 
\begin{align*}
\mathrel \bigg | \sum_{i= -\infty}^n a_i t^i\bigg|_{\infty} = q^n, \es \text{if} \es 0 \neq a_n \in \BF_q.
\end{align*}
In particular, if $f\in \BF_q[t]$, then 
$$
|f|_{\infty}=q^{\deg f}.
$$
The non-trivial additive character $E: \BF_q \to \BC^\ast$ is defined by 
\begin{align*}
E(x) = \exp \bigg\{\frac{2\pi i}{p} \Tr(x)\bigg\},
\end{align*}
and the map $e:  \BF_q(t)_\infty \to \BC^\ast$ is defined by 
\begin{align*}
e\bigg(\sum_{i= -\infty}^n a_i t^i\bigg) = E(a_{-1}).
\end{align*}
This map $e$ is a non-trivial additive character of $\BF_q(t)_\infty$. Furthermore, the additive characters modulo $f\in \BF_q[t]$ 
are precisely of the form 
$$
\psi: \BF_q[t] \longrightarrow \mathbb{C}^{\ast}, \quad \psi(r)= e\left(r\cdot \frac{g}{f}\right),
$$
where $g$ runs over a system of coset representatives of $(f)$ in $\BF_q[t]$
(see \cite{Hsu}). In analogy to Dirichlet characters, we define multiplicative 
characters modulo $f\in \BF_q[t]$ to be multiplicative maps
$$
\chi : \BF_q[t] \longrightarrow \mathbb{C}
$$
satisfying
$$
\chi(r)=0 \mbox{ if and only if } (r,f)\not=1  
$$
and 
$$
\chi(r_1)=\chi(r_2) \mbox{ if } r_1\equiv r_2 \bmod{f}.
$$
A multiplicative character modulo $f\in \BF_q[t]$ is referred to as primitive if it is not induced by a multiplicative character modulo a 
polynomial $g\in \BF_q[t]$ of smaller degree, i.e. if there does not exist a polynomial $g\in \BF_q[t]$ of smaller degree and a multiplicative
character $\tilde{\chi}$ modulo $g$ such that 
$$
\chi(r)=\tilde{\chi}(r) \mbox{ whenever } (r,f)=1.
$$
The principal character modulo $f\in \BF_q[t]$ is defined as 
$$
\chi_0(r):= \begin{cases} 1 & \mbox{ if } (r,f)=1, \\ 0 & \mbox{ if } (r,f)\not=1. \end{cases}
$$

\section{Large sieve inequalities for function fields}

\subsection{Large sieve with additive characters}
Below we recall the large sieve inequality for dimension 1 from the recent paper \cite{StephanRS} by the first and third-named 
authors. This will serve as the key tool in our present paper.

\begin{Theorem} \label{n=1} Let $Q$ and $N$ be positive integers, and $S$ be a set of non-zero monic polynomials in $\BF_q[t]$
of degree not exceeding $Q$. 
Then 
\begin{align*}
\sum_{f \in S} \ \sum_{\substack{r \bmod f,\\ (r,f)=1}}\mathrel \bigg 
|\sum_{\substack{g \in \BF_q[t]\\ \deg g\le N}} a_g  e\Big(g\cdot \frac{r}{f}\Big)\bigg |^2 \leq \es  
\left(q^{N+1} + \left(\sharp S\right) \cdot q^{Q-1}\right) \cdot  
\sum_{\substack{g \in \BF_q[t]\\ \deg g\le N}} |a_g|^2,
\end{align*}
where $a_g$ are arbitrary complex numbers. 
\end{Theorem}

\subsection{Large sieve with multiplicative characters}
Using Gauss sums, one can turn the large sieve for additive characters into the following large sieve for multiplicative characters. The proof
is completely analogue to the proof of Theorem 7 in \cite{Bom}, which states the large sieve for multiplicative characters in the classical 
case of integers. 

\begin{Theorem} \label{multipli} Under the conditions of Theorem {\rm \ref{n=1}}, we have   
\begin{align*}
\sum_{f \in S} \frac{q^{\deg f}}{\phi(f)}\ \sum_{\substack{\chi \bmod f,\\ \chi \ \text{\rm primitive}}}\mathrel \bigg 
|\sum_{\substack{g \in \BF_q[t]\\ \deg g\le N}} a_g  \chi(g)\bigg |^2 \leq \es  
\left(q^{N+1} + \left(\sharp S\right) \cdot q^{Q-1}\right) \cdot  
\sum_{\substack{g \in \BF_q[t]\\ \deg g\le N}} |a_g|^2,
\end{align*}
where $a_g$ are arbitrary complex numbers and $f$ is the Euler function for $\BF_q[t]$, defined as 
$$
\phi(f):=\sharp \{r \bmod f \ :\ (r,f)=1\}.
$$
\end{Theorem}

\begin{proof}
\end{proof}

\subsection{Arithmetic form of the large sieve}
Similarly as in the case of integers, the large sieve with additive characters for function fields 
can be turned into an upper bound sieve. The following 
result corresponds to Montgomery's \cite{Mon} arithmetic form of the large sieve in the classical setting of integers. 

\begin{Theorem} \label{arithls} Let $Q<N$ be positive integers, assume that
\begin{align*}
\FN \subseteq & \{f \in \BF_q[t]\ :\ f \ \text{\rm monic}, \ Q<\deg f \leq N\},\\
\FP \subseteq & \{ P \in  \BF_q[t]\ :\ P \ \text{\rm monic}, \ \deg P \leq  Q\}\ \text{\rm such that }\\ & 
P_1,P_2\in \FP\Longrightarrow (P_1,P_2)=1 \
\text{\rm or } P_1=P_2,\\
\Omega_P :  = & \DS \bigcup_{i=1}^{\Fw(P)}\CR_i(P) \ \text{\rm if } P \in \FP,\ \text{\rm where } \CR_i(P) \ (i=1,...,\Fw(P)) \\
& \text{\rm are distinct residue classes modulo } P,
\end{align*}
and set
\begin{equation*}
\FN^\ast := \{f\in \FN\ :\ f \notin \Omega_P\  \text{\rm for all} \ P \in \FP \}.
\end{equation*}
Further, set $\Fw(P)=0$ if $P\not\in \FP$. 
Define the set $\CK$ to be the union of $\{1\}$ and the set of all products of distinct elements of $\FP$, i.e.
\[
\CK\  :\  = \{1\} \cup \{P_1\cdots P_n\ :\ n\in \BN, \  P_1,...,P_n \ \text{are distinct elements of } \FP \}
\]
and 
\[
\CA_{\CK}(Q) =  \sharp\{k\in \CK\  :\  \deg k \leq Q\}.
\]
Then
\begin{equation} \label{desired}
\# \FN^\ast \leq \left(q^{N+1}+ \CA_{\CK}(Q) \cdot 2^Qq^{Q-1}\right)\bigg( \sum_{\substack{R\in \BF_q[t]\ \text{\rm monic} 
\\ \deg R\leq Q}}\Fg(R)\bigg) ^{-1},
\end{equation}
where the function $\Fg :  \BF_q[t] \rightarrow \mathbb{R}$ is defined by  
\begin{equation*}
\Fg(R) := \prod\limits_{P|R} \frac{\Fw(P)}{q^{\deg P}-\Fw(P)} 
\end{equation*}
if $R\in \CK$ and $\Fg(R)=0$ otherwise.
\end{Theorem}

To prove Theorem \ref{arithls}, we first introduce the following additional notations.  
If $\alpha \in \BF_q(t)_\infty$, we set
\[
S(\alpha) := \sum_{g\in \BF_q[t]} a_ g e(g\alpha), 
\]
where we put $a_g:=0$ if $\deg g>N$. We further set 
\[
\Fw'(P) :=\#\{ h\bmod{P}\  :\  a_g = 0 \ \text{for all}\ g \equiv h\bmod{P}  \}
\]
if $P \in \FP$. We define the function $\Fg' :  \BF_q[t] \rightarrow \mathbb{R}$ by  
\begin{equation}
\Fg'(R) \  :\  = \prod\limits_{P|R} \frac{\Fw'(P)}{q^{\deg P}-\Fw(P)} 
\end{equation}
if $R\in \CK$ and $\Fg'(R)=0$ otherwise. 
We first establish the following.

\begin{Lemma} \label{Lemma11} In the above notations, we have
\begin{equation} \label{lemmaeq}
\Fg'(R) \bigg|\sum_{g\in \BF_q[t]} a_g \bigg|^2 \leq \sum\limits_{\substack{r \bmod{R} \\  
P\nmid r \ \text{\rm if}\ P\in \FP \ \text{\rm and}\ P|R}} 
\left| S\Big(\frac{r}{R}\Big) \right|^2
\end{equation}
for every $R\in \BF_q[t]$.
\end{Lemma}

\begin{proof} The said inequality is trivial if $R\not\in \CK$. Therefore, we assume $R\in \CK$ throughout the following. We note that 
$$
\sum_{g\in \BF_q[t]} a_g= S(0),
$$
and therefore the claimed inequality \eqref{lemmaeq} is equivalent to 
\begin{align}\label{Ineq11}
\Fg'(R)\cdot |S(0)|^2 \leq \sum\limits_{\substack{r \bmod{R}\\ P\nmid r \ \text{\rm if}\ P\in \FP \ \text{\rm and}\ P|R}} \left| S\Big(\frac{r}{R}\Big) \right|^2.
\end{align}
Let $\beta \in \BF_q(t)_\infty$ be arbitrary. In the following, let us replace $a_g$ by $a_g e(g\beta)$ 
in the original definition of $S$. Then $\Fg'(R)$ doesn't change and hence \eqref{Ineq11} gives us 
\begin{align}\label{Ineq12}
\Fg'(R)\cdot |S(\beta)|^2 \leq \sum\limits_{\substack{r \bmod{R}\\ 
P\nmid r \ \text{\rm if}\ P\in \FP \ \text{\rm and}\ P|R}} \left| S\Big(\frac{r}{R}+\beta\Big) \right|^2
\end{align}
for all $\beta \in \BF_q(t)_\infty$. Now suppose that \eqref{Ineq11} holds for $R$ and $R'$ with $(R, R') =1$. 
Then inequalities \eqref{Ineq11} and  \eqref{Ineq12} above imply that the inequality \eqref{Ineq11} also holds for $RR'$ in place of $R$ because
\begin{align*}
 & \sum\limits_{\substack{r \bmod{RR'}\\ P\nmid r \ \text{\rm if}\ P\in \FP \ \text{\rm and}\ P|RR'}} 
 \left| S\Big(\frac{r}{RR'}\Big) \right|^2\\ = & 
 \sum\limits_{\substack{a \bmod{R}\\  P\nmid a \ \text{\rm if}\ P\in \FP \ \text{\rm and}\ P|R}}\ 
 \sum\limits_{\substack{a' \bmod{R'}\\  P\nmid a' \ \text{\rm if}\ P\in \FP \ \text{\rm and}\ P|R'}}\left| S\Big(\frac{a}{R}+\frac{a'}{R'}\Big) 
 \right|^2  \\
 \geq & \sum\limits_{\substack{a \bmod{R}\\  P\nmid a \ \text{\rm if}\ P\in \FP \ \text{\rm and}\ P|R}}\Fg'(R)\cdot
 \left| S\Big(\frac{a}{R}\Big) \right|^2 \\ \geq & 
 \ \Fg'(R)\Fg'(R')\cdot \left|S(0)\right|^2 \\ = &
 \ \Fg'(RR')\cdot \left|S(0)\right|^2.
\end{align*}
By these considerations it suffices to prove \eqref{Ineq11} for for $R = P\in\FP$. 
To this end, we look at 
\begin{equation} \label{look}
Z(P,h) := \sum_{\substack{g\in \BF_q[t]\\ g \equiv h \bmod{P}}} a_g \quad \mbox{and}\quad  Z := \sum_{g\in \BF_q[t]} a_g = S(0).
\end{equation}
Opening up the square and using orthogonality relations for additive characters, it is easy to calculate that 
\[
\sum\limits_{\substack{r\bmod{P}  }} \left| S\Big(\frac{r}{P}\Big) \right|^2 = q^{\deg P}\sum_{h \bmod{P}} 
\left|Z(P,h)\right|^2.
\]
Subtracting  $|S(0)|^2= |Z|^2$, we get 
\begin{equation} \label{subt}
\sum\limits_{\substack{r\bmod{P}\\ P\nmid r  }} \left| S\Big(\frac{r}{P}\Big) \right|^2 = q^{\deg P}\sum_{h \bmod{P}}
\left|Z(P,h)\right|^2- |Z|^2.
\end{equation}
Finally, we use the equation
$$ 
Z = \sum_{h \bmod{P}} Z(P,h)
$$ 
and the Cauchy-Schwarz inequality to deduce that 
\[
|Z|^2 \leq (q^{\deg P}-\Fw'(P))\sum_{h \bmod P} |Z(P,h)|^2. 
\]
Plugging this into \eqref{subt} and using the definition of $Z$ in \eqref{look} gives us the desired inequality 
\begin{align}
\Fg'(P) \mathrel\Big|\sum_{g \in \BF_q[t]} a_g \mathrel\Big |^2 \leq \sum\limits_{\substack{r \bmod{P} \\  P\nmid R}} 
\left| S\Big(\frac{r}{P}\Big) \right|^2.
\end{align}
\end{proof}

Now we are ready for the proof of Theorem \ref{arithls}. 

\begin{proof}
Summing \eqref{lemmaeq} over $R\in \mathcal{T}$ with 
$$
\mathcal{T}:=\{k\in \CK\ :\ \deg k\le Q\}
$$
and noting that $\Fg'(R)=0$ if $\deg R\le Q$ and $R\not\in \mathcal{T}$, we obtain 
\begin{equation} \label{bor1}
\begin{split}
\bigg(\sum_{\substack{R\in \BF_q[t]\ \text{\rm monic} 
\\ \deg R\leq Q}} \Fg'(R)\bigg)\cdot 
\bigg|\sum_{g\in \BF_q[t]} a_g \bigg|^2 \leq & \sum_{R \in \mathcal{T}}\ \sum\limits_{\substack{r \bmod{R} \\  
P\nmid r \ \text{\rm if}\ P\in \FP \ \text{\rm and}\ P|R}} 
\left| S\Big(\frac{r}{R}\Big) \right|^2\\
= & \sum_{\substack{D\not=1\ \text{\rm monic}\\ D|R \ \text{for some}\ R\in \mathcal{T}}}\ \sum\limits_{\substack{s \bmod{D} \\  
(s,D)=1}} 
\left| S\Big(\frac{s}{D}\Big) \right|^2.
\end{split}
\end{equation}
Using Theorem \ref{n=1}, the last line is bounded by  
\begin{equation} \label{bor2}
\begin{split}
\sum_{\substack{D\not=1\ \text{\rm monic}\\ D|R \ \text{for some}\ R\in \mathcal{T}}}\ \sum\limits_{\substack{s \bmod{D} \\  
(s,D)=1}} \left| S\Big(\frac{s}{D}\Big) \right|^2 \le \left(q^{N+1} + \sharp(S) \cdot
q^{Q-1}\right) \cdot \sum_{g \in \BF_q[t]}|a_g|^2,
\end{split}
\end{equation}
where 
$$
S:=\{D\not=1\ \text{monic} \ :\ D|R \ \text{for some}\ R\in \mathcal{T}\}.
$$
We note that 
\begin{equation} \label{bor3}
\sharp(S)\le 2^Q \sharp(\mathcal{T}) = 2^Q \CA_{\CK}(Q).  
\end{equation}
Combining \eqref{bor1}, \eqref{bor2} and \eqref{bor3}, we obtain
\begin{equation} \label{and}
\bigg(\sum_{\substack{R\in \BF_q[t]\ \text{\rm monic} 
\\ \deg R\leq Q}} \Fg'(R)\bigg)\cdot 
\bigg|\sum_{g\in \BF_q[t]} a_g \bigg|^2 \leq \left(q^{N+1} + \CA_{\CK}(Q) \cdot
2^Qq^{Q-1}\right) \sum_{g \in \BF_q[t]} |a_g|^2.
\end{equation}

Now we set
\begin{align*}
a_g & = 
\begin{cases}
 1 &  \text{if}\ g \in \FN^\ast,\\
0 &  \text{otherwise.}
\end{cases}
\end{align*}
Then $\Fw(P) \leq \Fw'(P)$ for all $P \in \FP$ and hence $\Fg'(R)\ge \Fg(R)$ for all $R\in \BF_q[t]$, and  
\[
\sum a_{g\in \BF_q[t]} = \sum_{g\in \BF_q[t]} | a_g|^2 = \# \FN^\ast.
\]
From this and \eqref{and}, we deduce the desired inequality \eqref{desired}. 
\end{proof}

\section{Proof of Theorem \ref{Psettheoff}}
Our proof depends on Theorem \ref{arithls}. We start with the following observation. 
If $P,  R \in \CS$ and $\deg P < \deg R$ then we have $U \not\equiv -R$ mod $ P$ 
for all $U \in \CS$ with $\deg U> \deg P$ (this follows directly from the definition of $\CS$ to be a $\CP$-set). 
From this, we conclude that for every $P\in \CS$ there exist at least $1 + [q^{\deg P}/2]$ residue classes 
$\CR_1(P), \CR_2(P),...,\CR_{\Fw(P)}(P)$ mod $P$ which do not contain any 
element of $\CS$ greater than $\deg P$. 

Now we use the following sieve. Let $Q<N$ be positive integers. Set 
\begin{align*}
\FN:= & \{f \in \mathcal{S}\ :\ Q<\deg f \leq N\},\\
\FP:= & \{ P \in  \mathcal{S}\ :\  \deg P \leq  Q\}, \\
\Omega_P \  :\  = & \DS \bigcup_{i=1}^{\Fw(P)}\CR_i(P) \ \text{\rm if } P \in \FP,\\
\FN^\ast := & \{f\in \FN\ :\ f \notin \Omega_P\  \text{\rm for all} \ P \in \FP \}.
\end{align*}
We note that $\FN=\FN^{\ast}$ in this case, by definition of $\CS$. It follows that
\begin{equation} \label{ob1}
\mathcal{A}_{\mathcal{S}}(N)\le \sharp\left(\FN^\ast\right)+q^Q
\end{equation}
and 
$$
\Fw(P)\geq 1 + [q^{\deg P}/2]
$$
if $P \in \FP$. Hence, in the notations of Theorem \ref{arithls}, we have 
\begin{equation} \label{ob2}
\Fg(R) \geq 1 \mbox{ if } R\in \mathcal{K}
\end{equation}
and 
\begin{equation} \label{ob3}
\Fg(R) =0 \mbox{ if } R\not\in \mathcal{K}.
\end{equation}
Using \eqref{desired}, \eqref{ob1}, \eqref{ob2} and \eqref{ob3}, we obtain
\[
\CA_{\CS}(N)  \leq 
\frac{ q^{N+1} + \mathcal{A}_{\mathcal{K}}(Q)\cdot 2^Qq^{Q-1}}{ \CA_{\CK}(Q) } + q^Q.
\]
Choosing $N \  :\  = Q+\lceil\log_q \mathcal{A}_{\mathcal{K}}(Q)\rceil$,  we deduce that 
\[
\CA_{\CS}\left(Q+\log_q \mathcal{A}_{\mathcal{K}}(Q)\right) \ll (2q)^Q .
\]
Assume that $0<\beta < 1$ such that    
$\CA_{\CS}(N)\geq q^{N\beta} $ for all $N$ large enough. Then it follows that 
\[
(2q)^Q\gg \CA_{\CS}\left(Q+\log_q \mathcal{A}_{\mathcal{K}}(Q)\right) \geq \CA_{\CS}\left(Q+\log_q \mathcal{A}_{\mathcal{S}}(Q)\right)
\ge \left(q^Q\cdot \mathcal{A}_{\mathcal{S}}(Q)\right)^\beta 
\geq q^{Q({\beta+\beta^2})}
\]
for all $Q$ large enough, which implies
$$
\beta +\beta^2 \le 1
$$
or, equivalently, $\beta  \leq  1/\Phi$, where $\Phi=(\sqrt{5}+1)/2$ is the golden ratio. Hence,
$$
\CA_{\CS}(N) \leq q^{(1/\Phi+\varepsilon)N}
$$ 
for infinitely many integers $N$. $\Box$

\section{Proof of Theorem \ref{ErSatheoff}}
Again, our proof depends on Theorem \ref{arithls}. We start with the following observation similar to that at the beginning of the last section. 
If $f,f' \in \CF$ and $P$ is an irreducible monic polynomial, then $f+f'\not\equiv 0\bmod{P^2}$. 
From this, we conclude that for every such $P$ there exist at least $1 + \left[q^{\deg P^2}/2\right]$ residue classes 
$\CR_1\left(P^2\right), \CR_2\left(P^2\right),...,\CR_{\Fw\left(P^2\right)}\left(P^2\right)$ mod $P^2$ which do not contain any 
element of $\CF$. 

Now we use the following sieve. Let $Q<N$ be positive integers. Set 
\begin{align*}
\FN:= & \{f \in \mathcal{F}\ :\ Q<\deg f \leq N\},\\
\FP:= & \{ P^2 \ :\  P \ \text{irreducible and monic and}\ \deg P^2 \leq  Q\}, \\
\Omega_{P^2} \  :\  = & \DS \bigcup_{i=1}^{\Fw\left(P^2\right)}\CR_i\left(P^2\right) \ \text{\rm if } P^2 \in \FP,\\
\FN^\ast := & \{f\in \FN\ :\ f \notin \Omega_{P^2}\  \text{\rm for all} \ P^2 \in \FP \}.
\end{align*}
By the observation at the beginning of this section, we have
\begin{equation} \label{obs1}
\sharp(\CF)\le \sharp\left(\FN^\ast\right)+q^{Q}
\end{equation}
and 
$$
\Fw\left(P^2\right)\geq 1 + \left[q^{\deg P^2}/2\right]
$$
if $P^2 \in \FP$. Hence, in the notation of Theorem \ref{arithls}, we have 
\begin{equation} \label{obs2}
\Fg(R) \geq 1 \mbox{ if } R\in \mathcal{K}
\end{equation}
and 
\begin{equation} \label{obs3}
\Fg(R) =0 \mbox{ if } R\not\in \mathcal{K}.
\end{equation}
Using \eqref{desired}, \eqref{obs1}, \eqref{obs2} and \eqref{obs3}, we obtain
\[
\sharp(\CF)  
\leq  
\frac{ q^{N+1} + \mathcal{A}_{\mathcal{K}}(Q)\cdot 2^Qq^{Q-1}}{ \CA_{\CK}(Q) } + q^Q.
\]
By the prime number theorem for function fields, we have 
$$
\mathcal{A}_{\mathcal{K}}(Q)\ge \sharp\{P\in \BF_q[t]\ :\ P \ \text{monic and irreducible and } \deg P \le Q/2\} \gg \frac{q^{Q/2}}{Q}.
$$
Choosing $Q:=\lceil N/3 \rceil$,  we deduce that 
\[
\sharp(\CF) \ll q^{N(2/3+\varepsilon)},
\]
as claimed. $\Box$

\section{Proof of Theorem \ref{SarRitheoff}}

\begin{proof} Assume that, contrary to the statement of Theorem \ref{SarRitheoff}, we have 
\begin{align}\label{ContMV}
\min\{\sharp(\mathcal{F}),\sharp(\mathcal{G})\}> q^{N(2/3+\varepsilon)} 
\end{align}
for $N$ sufficiently large, but $fg+1$ is square free for all $f \in \CF, g \in \CG$.
Then, by the orthogonality relations for multiplicative characters,  we have 
\begin{align*} 
0 &= \sharp\left\{ (f,g)\in \CF\times \CG\ :\  fg+1 \equiv 0 \bmod{P^2}\right\}  \\
& = \frac{1}{\phi(P^2)} \sum_{\chi\bmod{P^2}}\overline \chi(-1)\sum_{f\in \CF} \sum_{g\in \CG}\chi(fg)
\end{align*}
for every irreducible monic polynomial $P$. This implies 
\begin{equation*}
\begin{split}
\sharp\{ (f,g)\in \CF\times \CG\ :\ (fg, P)=1 \} = & \chi_0(-1)\sum_{f\in \CF}\sum_{g\in \CG}\chi_0(fg) \\
= & -  \sum_{\substack{\chi\bmod{P^2} \\ \chi \neq \chi_0}}\overline \chi(-1)\sum_{f\in \CF} \sum_{g\in \CG}\chi(fg),
\end{split}
\end{equation*}
where $\chi_0$ is the principal character mod $P^2$. Hence, 
\begin{equation} \label{ineq}
\sharp\{ (f,g)\in \CF\times \CG\ :\ (fg, P)=1 \}\le   
\sum_{\substack{\chi \bmod{P^2} \\ \chi \neq \chi_0}}\Big|\sum_{f\in \CF} \chi(f)\Big|\cdot \Big| \sum_{g\in \CG}\chi(g)\Big|.
\end{equation}
Now we set $Q = \lceil N/3\rceil$ and 
$$
S:=\{P\in \BF_q[t]\ :\ P \ \text{irreducible and monic and } \deg P=Q\}.
$$
Summing \eqref{ineq} over all irreducible monic polynomials $P\in S$, we get
\begin{align*}
\sS := & \sum_{P\in S} \sharp\{ (f,g)\in \CF\times \CG\ :\ (fg, P)=1 \} \\
     \le & \sum_{P\in S}\ \sum_{\substack{\chi\bmod{P^2} \\ \chi \neq \chi_0}}
     \Big|\sum_{f\in \CF} \chi(f)\Big|\cdot \Big| \sum_{g\in \CG}\chi(g)\Big|.\nonumber
\end{align*}
Using the  Cauchy-Schwarz inequality, it follows that 
\begin{align}\label{IneqY}
\sS \le \sS_{\CF}^{\frac{1}{2}}\sS_{\CG}^{\frac{1}{2}},
\end{align}
where 
\begin{equation} \label{sCFdef}
\sS_{\CF} =   \sum_{P \in S} \sum_{\substack{\chi\bmod{P^2} \\ \chi \neq \chi_0}}\mathrel \Big|\sum_{f\in \CF} \chi(f)\mathrel \Big|^2
\end{equation}
and $\sS_{\CG}$  is defined similarly with ${\CG}$ in place of $\CF$.
On the other hand, using \eqref{ContMV}, our choice $Q=\lceil N/3 \rceil$ and 
\begin{equation} \label{prime}
\sharp(S)\gg \frac{q^Q}{Q},
\end{equation}
which is a consequence of the prime number theorem for function fields, we have 
\begin{equation}\label{IneqX}
\begin{split}
\sS  \ge & \sum_{P\in S} (\sharp\{ (f,g)\in \CF\times \CG \} - 
     \sharp\{ (f,g) \in \CF\times \CG\ :\  P|f\}-\\ 
              & \ \ \ \ \ \ \ \sharp\{ (f,g)\in \CF\times \CG\ :\  P|g\}) \\
              \ge & \left(\sharp(\CF)\cdot \sharp(\CG) - \left(\sharp(\CF)+ \sharp(\CG)\right)q^{N-Q}\right)\cdot \sharp(S)\\
              \gg & \sharp(\CF)\cdot \sharp(\CG)\cdot  q^{N(1/3-\varepsilon/2)},
\end{split}
\end{equation}
where the last line arrives by using \eqref{ContMV}.
Further, we may rewrite the sum $\sS_{\CF}$ defined in \eqref{sCFdef} in the form 
\begin{align*}
\sS_{\CF} &=   \sum_{P \in S} \bigg(\sum_{\substack{\chi\bmod{P^2} \\ \chi\ \text{primitive}}}\mathrel \Big|\sum_{f\in \CF} \chi(f)\mathrel \Big|^2 + \sum_{\substack{\chi\bmod{P} \\ 
\chi\ \text{primitive}}}\mathrel \Big|\sum_{f\in \CF} \chi(f)\mathrel \Big|^2 \bigg). 
\end{align*}
Using Theorem \ref{multipli}, $Q=\lceil N/3 \rceil$ and \eqref{prime}, we estimate the right-hand side to obtain 
\begin{align}\label{IneqZ}
\sS_{\CF} \le  \es  \left(q^{N+1} + 2\left(\sharp S\right)q^{2Q-1} \right)\cdot \sharp(\CF) \ll q^N \cdot \sharp(\CF),
\end{align}
and in the same way we get
\begin{align}\label{IneqW}
\sS_{\CG}\ll q^N \cdot \sharp(\CG).
\end{align}
It follows from \eqref{IneqY}, \eqref{IneqX}, \eqref{IneqZ} and \eqref{IneqW} that
\begin{align*}
\sharp(\CF)\cdot \sharp(\CG)\cdot  q^{N(1/3-\varepsilon/2)} \ll \sS \ll q^N  \cdot \left(\sharp(\CF)\cdot \sharp(\CG)\right)^{1/2},
\end{align*}
which implies that
\[
\min\{\sharp(\CF),\sharp(\CG)\}\le \left(\sharp(\CF)\cdot \sharp(\CG)\right)^{1/2} \ll q^{N(2/3+\varepsilon/2)},
\] 
contradicting \eqref{ContMV}. 
\end{proof}

\end{document}